\documentclass[11pt,reqno,a4paper]{amsart}
\usepackage{hyperref}
\usepackage{latexsym}
\usepackage{amssymb,amsmath}
\usepackage[usenames,dvipsnames]{xcolor}
\usepackage{url}
\usepackage{booktabs}
\usepackage{enumerate}
\usepackage[shortlabels]{enumitem}
\usepackage{lastpage}
\usepackage{graphicx}
\usepackage{varwidth}
\usepackage{amsthm}
\usepackage{mathrsfs}
\usepackage{amsfonts,amscd}
\usepackage{youngtab}
\allowdisplaybreaks

\usepackage{amsthm}
\newtheorem{thm}{Theorem}[section]
\newtheorem{lem}[thm]{Lemma}
\newtheorem{cor}[thm]{Corollary}

\newtheorem{defi}[thm]{Definition}

\newtheorem*{thmA}{Theorem A}
\newtheorem*{thmB}{Theorem B}

\newtheorem*{thm*}{Theorem}

\theoremstyle{definition}

\newtheorem{ex}[thm]{Example}

\numberwithin{equation}{section}

\def\irr#1{{\Irr}(#1)}

\def\oh#1#2{{\bf O}_{#1}(#2)}

\def\syl#1#2{{\rm Syl}_{#1}(#2)}

\def\nor{\triangleleft\,}

\def\norm#1#2{{\bf N}_{#1}(#2)}
\def\cent#1#2{{\bf C}_{#1}(#2)}
\def\aut#1{{\rm Aut}(#1)}

\let\phi=\varphi

\def\C{\mathbb C}

\def\Z{\mathbb Z}

\def\N{\mathbb N}

\newcommand{\Irr}{\operatorname{Irr}}
\newcommand{\height}{\operatorname{ht}}
\newcommand{\ch}{\operatorname{ch}}
\newcommand{\fS}{{\mathfrak{S}}}
\newcommand{\fA}{{\mathfrak{A}}}

\begin{document}

\author{Eugenio Giannelli}
\address[E. Giannelli]{Dipartimento di Matematica e Informatica U.~Dini, Viale Morgagni 67/a, Firenze, Italy}
\email{eugenio.giannelli@unifi.it}

\author{Stacey Law}
\address[S. Law]{Department of Pure Mathematics and Mathematical Statistics, University of Cambridge, Cambridge CB3 0WB, UK}
\email{swcl2@cam.ac.uk}

\author{Jason Long}
\address[J. Long]{Mathematical Institute, University of Oxford, Andrew Wiles Building, Radcliffe Observatory Quarter, Woodstock Road, Oxford OX2 6GG, UK} 
\email{jasonlong272@gmail.com}

\author{Carolina Vallejo}
\address[C. Vallejo]{Departamento de Matem\'aticas, Edificio Sabatini, Universidad Carlos III de Madrid,
Avenida Universidad 30, 28911, Legan\'es. Madrid, Spain}
\email{carolina.vallejo@uc3m.es}
\thanks{}

\title[]{Sylow branching coefficients and a conjecture of Malle and Navarro}


\begin{abstract} 
	We prove that a finite group $G$ has a normal Sylow $p$-subgroup $P$ if, and only if, every irreducible character of $G$ appearing in the permutation character $({\bf 1}_P)^G$ with multiplicity coprime to $p$ has degree coprime to $p$. 
	This confirms a prediction by Malle and Navarro from 2012. Our proof of the above result depends on a reduction to simple groups and ultimately on a combinatorial analysis of the properties of Sylow branching coefficients for symmetric groups.
\end{abstract}


\maketitle


\section{Introduction}

\noindent One of the main research themes in the representation theory of finite groups is to determine how much information about the algebraic structure of a finite group $G$ can be discovered using knowledge of its character degrees. A famous result in this line of investigation is the It\^o-Michler theorem \cite{Ito51,Michler} asserting that a Sylow $p$-subgroup $P$ of $G$ is abelian and normal in $G$ if, and only if, the character degree $\chi(1)$ is coprime to $p$ for every irreducible character $\chi\in\irr G$. 
Separating the two conditions (abelian and normal) on the Sylow $p$-subgroup $P$ in the context of character degrees has been a challenge for the last few decades. While the commutativity of $P$ is characterized by Brauer's height zero conjecture \cite{MN21}, the aim of this article is to study canonical subsets of characters whose degrees characterize the normality of $P$ in $G$. 

In \cite{MN12} Malle and Navarro showed that given a prime $p$ and a Sylow $p$-subgroup $P$ of $G$, then $P$ is normal in $G$ if, and only if, every irreducible constituent of the permutation character $({\bf 1}_P)^G$ has degree coprime to $p$. At the end of their article they conjecture a refinement of this result, proposing that the normality of $P$ may be detected by looking at an even smaller subset of irreducible characters of $G$, namely those irreducible constituents of $({\bf 1}_P)^G$ appearing with multiplicity coprime to $p$. 
In this article we verify Malle and Navarro's prediction.

\begin{thmA} Let $G$ be a finite group, $p$ be a prime and $P\in \syl p G$. The following statements are equivalent:
\begin{enumerate}
\item[{\rm (i)}] $P$ is normal in $G$.
\item[{\rm (ii)}] Every $\chi \in \irr G$ with $[\chi_P, {\bf 1}_P]$ not divisible by $p$ has degree coprime to $p$.
\item[{\rm (iii)}] Every $\chi \in \irr G$ with $[\chi_P, {\bf 1}_P]$ not divisible by $p$ does not vanish in $P$.
\end{enumerate}
\end{thmA}

In Section 2, we show that in order to prove Theorem A it is enough to prove the validity of its statement for all finite simple non-abelian groups. Roughly speaking, for every finite non-abelian simple group $S$ we must exhibit
an irreducible character $\chi$ of degree divisible by $p$ with {\it trivial Sylow branching coefficient} coprime to $p$. If $P$ is a Sylow $p$-subgroup of $S$, then the trivial Sylow branching coefficient of $\chi$ is the multiplicity $[\chi_P, {\bf 1}_P]$ with which $\chi$ appears as a constituent of the permutation character $({\bf 1}_P)^S$. (We refer the reader to Section \ref{sec:311} for more information on Sylow branching coefficients.) The main obstacle in this context comes from simple alternating groups (as already observed in \cite[p.4]{MN12} and recently remarked by the same authors in \cite{MN2}). We prove the statement of Theorem A for simple alternating groups as a consequence of the following much more general statement, concerning symmetric and alternating groups,  $\fS_n$ and $\fA_n$, at all primes.

\begin{thmB}
	Let $p$ be a prime and $n\in\N$. Let $G\in\{\fS_n, \fA_n\}$, let $P$ be a Sylow $p$-subgroup of $G$ and let $B$ be a $p$-block of $G$. Then there exists an irreducible character $\chi$ of height zero in $B$ such that $[\chi_P, {\bf 1}_P]$ is not divisible by $p$.
\end{thmB}

As well as providing the key ingredient for proving Theorem A (note that the character $\chi$ provided by Theorem B has degree divisible by $p$ whenever $B$ has non-maximal defect), Theorem B also contributes to
the study of Sylow branching coefficients for symmetric and alternating groups and, more generally, to the study of the restriction of characters to Sylow subgroups.
These topics have recently been at the centre of investigation for their connections to the McKay Conjecture \cite{NTV, GKNT, INOT, GN}. In \cite{GL1} the authors determine those irreducible characters of $\fS_n$ having non-zero trivial Sylow branching coefficient.
Despite this positivity result, very little is known about the values of these integers. 
In this sense, Theorem B represents a first step towards a more precise description of these Sylow branching coefficients.

The key idea behind the proof of Theorem B is the following. For any given $p$-block $B$ of $\fS_n$ we introduce a virtual character $V^{B}$, obtained as a certain integer combination of the irreducible characters of height zero in $B$ (see Definition \ref{def: VirtualCharBlock}). Using the language of symmetric functions together with algebraic-combinatorial techniques, we then show that $p$ does not divide the multiplicity $[(V^{B})_P, {\bf 1}_P]$, and hence deduce that $p$ does not divide the Sylow branching coefficient corresponding to one of the height zero characters occurring in $V^{B}$. 

It is worth mentioning that the virtual characters $V^B$, and more generally the family of virtual characters introduced in Section~\ref{sec:coprime} below, seem to have further applications to problems involving signed character sums in symmetric and alternating groups (see \cite[p.2]{IN08} and \cite[Section 6]{N10}); this will be the subject of future investigation.

The structure of this article is as follows. In Section~\ref{sec:MN} we prove Theorem A, assuming that its statement holds for simple alternating groups. 
In Section~\ref{sec:coprime} we investigate symmetric and alternating groups specifically, in particular showing that Theorem A holds for these classes of groups as a consequence of the more general Theorem B.

\medskip

\section{A reduction to alternating simple groups}\label{sec:MN} 
\noindent The aim of this section is to prove Theorem A.
%
%
%
We mimic and adapt the approach used in \cite{MN12}: we reduce the problem to showing that every finite non-abelian simple group $S$
possesses a suitable character lying over the trivial character of a Sylow $p$-subgroup of $S$ with multiplicity not divisible by $p$ and
vanishing on some element of the aforementioned Sylow $p$-subgroup.

\medskip

We follow the notation of \cite{Isa06} and \cite{Nav18} for characters. Let $G$ be a finite group and $p$ be a prime. Recall that $\chi \in \irr G$ has {\em $p$-defect zero} if the $p$-part of its degree $\chi(1)$ is as large as possible, that is, if $p$ does not divide $|G|/\chi(1)$.
Note that a non-trivial group with a normal Sylow $p$-subgroup $P$ does not possess a $p$-defect zero character unless 
$P=1$. 

\begin{lem}\label{lem:defectzero} 
	Suppose that $\chi \in \irr G$ has $p$-defect zero and let $P\in \syl p G$. Then
	\begin{enumerate}
	\item[{\rm (i)}] $\chi(x)=0$ for every non-trivial $x \in P$.
	\item[{\rm (ii)}] $\chi_P = f\cdot\rho_P$ where $f$ is coprime to $p$ and $\rho_P$ is the regular character of $P$.
	\end{enumerate}
\end{lem}

\begin{proof}
	Part (i) is \cite[Theorem 8.17]{Isa06} (also \cite[Theorem 4.6]{Nav18}). By (i), we have that
	 $$[\chi_P, {\bf 1}_P]=\frac{1}{|P|}\sum_{x \in P} \chi(x)=\frac{\chi(1)}{|P|}=f$$ is a positive integer. Then $\chi_P=f\cdot\rho_P$ and part (ii) follows.
\end{proof}
 
The next lemma follows from a standard argument. 
 
\begin{lem}\label{lem:nonvanishing2} 
	Let $G$ be a finite group, $p$ be a prime and $P\in \syl p G$.
	Let $\chi \in \irr G$ have degree coprime to $p$. Then $\chi(x)\neq 0$ for every $x \in P$.
\end{lem}

\begin{proof} 
	Consider the ring $\bf R$ of algebraic integers of $\C$. Let $\mathcal M$ be a maximal ideal of $\bf R$ containing $p\bf R$.
	Let $\xi \in \bf R$ be a root of unity of order $p^a$. Then $(\xi-1)^{p^a}\equiv 0 \mod \mathcal M$. Since $\bf R/\mathcal M$ is a field we obtain that $\xi \equiv 1 \mod \mathcal M$. Given $x \in P$, write $Q=\langle x \rangle$. Since $\chi_Q=\lambda_1 +\cdots +\lambda_{\chi(1)}$ where $\lambda_ j \in \irr Q$, then $\chi(x)$ is a sum of $\chi(1)$ roots of $p$-power order and 
	$$\chi(x)\equiv \chi(1) \mod \mathcal M\, .$$
	In particular, $\chi(x)\neq 0$ as $\mathcal M\cap \mathbb Z=p \mathbb Z$.
\end{proof}

It is a well-known result of Burnside that every non-linear character of a finite group vanishes on some element (\cite[Corollary 4.2]{Nav18}). Hence the converse of the above lemma holds in $p$-groups. However, it is not the case in general that the non-vanishing property on Sylow $p$-subgroups characterises characters of degree coprime to $p$, as shown by ${\rm SL}_2(5)$ for $p=2$. 
A key ingredient in the proof of Theorem A will be that symmetric and alternating groups satisfy the converse of 
Lemma \ref{lem:nonvanishing2} (see Theorem~\ref{thm:nonvanishing} below).

\medskip

We will prove Theorem A of the introduction using the following result on finite non-abelian simple groups.

\begin{thm}\label{thm:simplegroups} 
	Let $S$ be a finite non-abelian simple group of order divisible by a prime $p$, and let $R\in \syl p S$. Then $S$ either possesses a $p$-defect zero character or there is some  $\aut S$-invariant $\theta \in \irr S$ such that $[\theta_R, {\bf 1}_R]$ is not divisible by $p$ and $\theta(x)=0$ for some $x \in R$.
\end{thm}

\begin{proof} 
	By \cite[Corollary 2]{GO96}, every finite non-abelian simple group possesses a $p$-defect zero character unless $p=2$ and $S$ is one of the sporadic groups ${\rm M}_{12}$, ${\rm M}_{22}$, ${\rm M}_{24}$, ${\rm J}_2$, ${\rm HS}$, ${\rm Suz}$, ${\rm Ru}$, ${\rm Co}_3$, ${\rm Co}_1$, ${\rm BM}$ or an alternating group $\fA_n$ with $7\leq n\neq 2m^2+m$, $2m^2+m+2$ for any integer $m$; or $p=3$ and $S$ is either ${\rm Suz}$, ${\rm Co}_3$ or $\fA_n$ with $3n + 1$ divisible by some prime $q$ congruent to $2\pmod 3$ to an odd power. 
	For $p=2$, let $S$ be a sporadic group not admitting a $p$-defect zero character. Using \cite{GAP} and the command \texttt{PermChars(CharacterTable("S"), d))}, we can compute
	the permutation characters $({\bf 1}_R)^S$ for 
	$S \in \{ {\rm M}_{12}, {\rm M}_{22}, {\rm M}_{24}, {\rm J}_2, {\rm HS}\}$, where the second argument \texttt{d} is the degree of the desired permutation character. For $S={\rm BM}$ the character $({\bf 1}_R)^S$ was computed by T. Breuer\footnote{ \tiny See \url{http://www.math.rwth-aachen.de/homes/sam/ctbllib/doc2/chap8_mj.html\#X87D11B097D95D027}.}.
	
	For $S \in \{ {\rm Suz},{\rm Ru}, {\rm Co}_1,{\rm Co}_3\}$, one can compute $({\bf 1}_R)^S$ by choosing a maximal subgroup $M$
	of $S$ containing $R$ with the \cite{GAP} command\linebreak \texttt{Maxes(CharacterTable("S"))},  computing $\theta:=({\bf 1}_R)^M$ with \texttt{PermChars}, and finally inducing $\theta$ to $S$. (In the case where $S={\rm Co}_1$, choose ${\rm Co}_3$ as a maximal subgroup.)
	One can proceed in a similar way to obtain $({\bf 1}_R)^S$ whenever $S \in \{ {\rm Suz}, {\rm Co}_3 \}$ and $p=3$. 
	(The function \texttt{PermChars} has several strategies to determine candidates for permutation characters, and the second argument determines which one is chosen. In the case where $S={\rm Co}_3$ and $p=3$, one should use \texttt{PermChars(CharacterTable("S"), rec(torso:=[d]))} where \texttt{d} is the degree of the desired permutation character.)

	Once we have $({\bf 1}_R)^S$ stored in \cite{GAP}, one can easily check that there is some $\aut S$-invariant $\theta \in \irr S$ with $[\theta_R, {\bf 1}_R]$ coprime to $p$ that vanishes on some $p$-power order element. 
	
	It remains to find $\theta$ for the alternating groups $\fA_n$ with $n\ge 5$ and $p\in\{2,3\}$. This is given by Theorem~\ref{thm:An23} below.
\end{proof}

The following technical lemma will be useful for proving Theorem A.

\begin{lem}\label{lem:defectzerocase} 
	Let $G$ be a finite group, $p$ be a prime and $P\in \syl p G$. Suppose that $N\nor G$ is such that $PN\nor G$ and $Q=P\cap N>1$.
	\begin{enumerate}
	\item[{\rm (a)}]  If $\tau \in \irr {PN}$ lies over ${\bf 1}_P$ with multiplicity coprime to $p$, then so does any $\chi \in \irr G$ lying over $\tau$.
	\item[{\rm (b)}] If $N$ has a $p$-defect zero character $\eta$, then $G$ has an irreducible character $\chi \neq {\bf 1}_G$ such that $p$ does not divide $[\chi_P, {\bf 1}_P]$ and $\chi$ vanishes on the non-trivial elements of $Q$.
	\end{enumerate}
\end{lem}

\begin{proof} (a)  By assumption $m_\tau=[\tau_P, {\bf 1}_P]$ is coprime to $p$. Write
$\chi_{PN}=e\sum_{i= 1}^t\tau^{x_i}$ by Clifford's theorem \cite[Theorem 6.2]{Isa06} with $x_i \in G$. By the Frattini argument $G=N\norm G P$ and we can choose $x_i \in \norm G P$ (with $x_1=1$). By \cite[Corollary 11.29]{Isa06}, $\chi(1)/\tau(1)=et$ divides $|G:PN|$, and hence $e$ and $t$ are coprime to $p$. In particular, $\chi_P=e\sum_{i=1}^t(\tau^{x_i})_P=e\sum_{i=1}^t(\tau_P)^{x_i}$ and so $[\chi_P, {\bf 1}_P]=etm_\tau$ is coprime to $p$.

(b) Note that $Q=P\cap N \in \syl p N$. In particular, $\eta_Q=f\cdot \rho_Q$ by Lemma \ref{lem:defectzero}(b). Notice that $(\eta^{PN})_P=(f\cdot \rho_Q)^P=f\cdot \rho_P$, so that $p$ does not divide $f=[(\eta^{PN})_P, {\bf 1}_P]$. We can choose $\tau \in \irr{PN}$ lying over $\eta$ and ${\bf 1}_P$, and such that $m_\tau=[\tau_P, {\bf 1}_P]$ is not divisible by $p$. Let $\chi \in \irr G$ lie over $\tau$. By part (a), the multiplicity $[\chi_P, {\bf 1}_P]$ is coprime to $p$.
Moreover, $\chi_N$ is a multiple of the sum of the $\norm G P$-conjugates of $\eta$, and hence a sum of $p$-defect zero characters of $N$. In particular, by Lemma \ref{lem:defectzero}(a), $\chi$ vanishes on every non-trivial element of $Q$.
\end{proof}

\begin{proof}[Proof of Theorem A]
	(i) $\Rightarrow$ (ii): If $P\nor G$ and $\chi \in \irr{({\bf 1}_P)^G}$ then $\chi$ can be seen as a character of $G/P$ and hence has degree coprime to $p$.

	(ii) $\Rightarrow$ (iii): This implication follows from Lemma \ref{lem:nonvanishing2}.

	(iii) $\Rightarrow$ (i): Suppose that $G$ is a counterexample to the statement of minimal order. 
	Of course $G>1$ and $\norm G P <G$. Let $1<M \nor G$. Given $\chi \in \irr{G/M}$ lying above ${\bf 1}_{PM/M}$ with multiplicity coprime to $p$, we can view $\chi$ as an irreducible character of $G$, and then $\chi_{PM}=m {\bf 1}_{PM}+\Delta$ where $p$ does not divide $m$ and $[\Delta , {\bf 1}_{PM}]=0$. 
	Note that every irreducible constituent of $\chi_{PM}$ contains $M$ in its kernel, and hence restricts irreducibly to $P$. In particular, $[\chi_P, {\bf 1}_P]=m$ is not divisible by $p$. By assumption, $\chi$ does not vanish in $P$ (so in $PM/M$ as a character of $G/M$). By minimality of $G$, we conclude that $PM\nor G$. Hence ${\bf O}_ p(G)=1$. 

	Let $N$ be a minimal normal subgroup of $G$ and write $Q=P\cap N\in \syl p N$.  By the paragraph above we have that $PN\nor G$. 
\smallskip

\noindent {\bf Case 1:} Suppose that $p$ divides the order of $N$. Since $N$ is not a $p$-group because $\oh p G=1$, we have that $N$ is semisimple. Let $S \nor N$ be a minimal normal subgroup of $G$. Then $N=\prod_{i=1}^r S^{g_i}$  where $\{ S^{g_i}\}_{i=1}^r=\{ S^g \ | \ g \in G\}$ and we may assume $g_1=1$. In fact, $N=\times_{i=1}^rS^{g_i}$.
	Write $R=Q\cap S\in \syl p S$, and note that $R>1$.
	Since $S$ is a non-abelian simple group, by Theorem~\ref{thm:simplegroups} either (a) $S$ has a $p$-defect zero character $\theta$ or (b) $S$ has an $\aut S$-invariant $\theta \in \irr S$ such that $p$ does not divide $[\theta_R, {\bf 1}_R]$ and $\theta(x)=0$ for some $x \in R$.
	
	In case (a), note that $\eta=\times_{i=1}^r\theta^{g_i}\in \irr N$ has $p$-defect zero. Then Lemma \ref{lem:defectzerocase}(b) yields a contradiction.
	In case (b), write $\eta=\times_{i=1}^r\theta^{g_i}\in \irr N$. Note that $\eta$ is $G$-invariant as $\theta$ is $\aut S$-invariant. Moreover,
	$[\eta_Q, {\bf 1}_Q]=[\theta_R, {\bf 1}_R]^r$ is not divisible by $p$ and $\eta(y)=0$ where $y=\prod_{i=1}^r x^{g_i} \in Q$. Since
	$(\eta^{PN})_P=(\eta_Q)^P$ contains ${\bf 1}_P$ with multiplicity $[\eta_Q, {\bf 1}_Q]$ coprime to $p$, we can choose $\tau \in \irr {PN}$ lying over $\eta$ such
	that $p$ does not divide $[\tau_P, {\bf 1}_P]$. Let $\chi \in \irr G$ lie over $\tau$. By Lemma \ref{lem:defectzerocase}(a) we have that $p$ does not divide 
	$[\chi_P, {\bf 1}_P]$. As $\eta$ is $G$-invariant, hence $\chi_N=e\eta$ and so $\chi(y)=0$ for $y \in Q\subseteq P$, yielding a contradiction also in this case.
	
\smallskip
	
\noindent {\bf Case 2:} We are left to deal with the case where $N$ is a $p'$-group. Take $K/N$ a minimal normal subgroup of $G/N$ with $K\subseteq PN$. Write $Q=P\cap K\in \syl p K$. In particular, $K=NQ$ and $Q\cong K/N$ is a $p$-elementary abelian group. By the Frattini argument $G=K\norm G Q=N \norm G Q$, and therefore it is easy to see that $\cent Q N\nor G$. The fact that ${\bf O}_p (G)=1$ forces $\cent Q N$ to be trivial, and consequently the action of $Q$ on $N$ is faithful.
 	By \cite[Lemma 2.8]{DPSS09} there is some $\theta \in \irr N$ with $K_\theta=N$. (Note that the hypotheses of \cite[Lemma 2.8]{DPSS09} are fulfilled as $Q$ acts coprimely and faithfully on $N$, $Q$ is abelian and $N$ is characteristically simple.) Let $\eta=\theta^K\in \irr K$. Then $\eta$ has $p$-defect zero as a character of $K$ and Lemma \ref{lem:defectzerocase}(b) yields the final contradiction.
\end{proof}

\section{Sylow branching coefficients of $\fS_n$ and $\fA_n$}\label{sec:coprime}

\noindent The main aim of this section is to prove Theorem B. Using Theorem B, we then complete the proof of Theorem A by showing that Theorem~\ref{thm:simplegroups} holds for alternating groups at the primes $2$ and $3$. This is done 
in Theorem~\ref{thm:An23}.

We start by recording some notation and standard facts that will be used throughout this section.

\subsection{Preliminaries}\label{sec:prelim}
For $m$ a natural number we denote by $[m]$ the set $\{1,2,\dotsc,m\}$. For $p$ a prime, $\nu_p(m)$ denotes the $p$-adic valuation of $m$, i.e.~$m=p^{\nu_p(m)}t$ where $p\nmid t$.
For a finite group $G$ and a prime number $p$, we write $\Irr_{p'}(G)=\{\chi\in\Irr(G) : p\nmid\chi(1)\}$ for the set of irreducible characters of $G$ of degree coprime to $p$. For a $p$-block $B$ of $G$, let $\Irr_0(B)$ denote the set of height zero characters in $B$. Recall that if $B$ has defect group $D$ then the height $\height(\chi)$ of an irreducible character $\chi$ is given by $\height(\chi) = \nu_p\big(\chi(1)\big) + \nu_p(|D|) - \nu_p(|G|)$. 
As is customary we denote by $g^G$ the conjugacy class of the element $g$ in $G$.

We start by recording a group-theoretical result that will be used in the proof of Theorem B.
\begin{lem}\label{lem:iv}
	Let $G$ be a finite group, let $p$ be a prime and let $P\in\syl{p}{G}$. Let $g$ be an element of $P$.  Then $\nu_p\big(|P\cap g^G|\big) = \nu_p\big(|g^G|\big)$.
\end{lem}

\begin{proof}
From the definition of induced character \cite[Chapter 5]{Isa06}, we have $$({\bf 1}_P)^G(g)=\frac{|\cent{G}{g}|\cdot |g^G\cap P|}{|P|}=\frac{|G:P|\cdot |g^G\cap P|}{|g^G|}.$$
We then observe that $({\bf 1}_P)^G(g)$ equals the number of fixed left cosets of $P$ in $G$ under the action of $\langle g \rangle$ by left multiplication. It follows that 
 $({\bf 1}_P)^G(g)$ is coprime to $p$. The statement then follows.
\end{proof}

\subsubsection{Characters and combinatorics of $\fS_n$}\label{sec:311}
We let $\mathcal{P}(n)$ denote the set of partitions of $n$. Given $\lambda\in\mathcal{P}(n)$ (also written $\lambda\vdash n$) we denote its conjugate by $\lambda'$. 
The set $\Irr(\fS_n)$ of ordinary irreducible characters of $\fS_n$ is naturally in bijection with $\mathcal{P}(n)$. For a partition $\lambda\in\mathcal{P}(n)$, we denote the corresponding irreducible character of $\fS_n$ by $\chi^\lambda$. 
Given $P_n$ a Sylow $p$-subgroup of $\fS_n$ and $\phi\in\Irr(P_n)$, we use the notation introduced in \cite{GL2} by letting $Z^\lambda_\phi$ denote the natural number defined by
\[ Z^\lambda_\phi:=[(\chi^\lambda)_{P_n},\phi]. \]
These multiplicities are called Sylow branching coefficients for symmetric groups. 
In this article we will be particularly interested in the case where $\phi={\bf 1}_{P_n}$ is the trivial character of $P_n$. We will sometimes use the symbol $Z^\lambda$ to denote $Z^\lambda_{{\bf 1}_{P_n}}$, to ease the notation.

To each partition $\lambda=(\lambda_1,\dotsc,\lambda_k)$ we may associate a Young diagram given by $[\lambda]=\{(i,j)\in\N\times\N : 1\le i\le k,\ 1\le j\le \lambda_i \}$. The hook of $\lambda$ corresponding to the node $(i,j)$ is denoted by $h_{i,j}(\lambda)$ and we let $|h_{i,j}(\lambda)|$ denote its size. 

For any $e\in\N$, we denote by $C_e(\lambda)$ the $e$-core of the partition $\lambda$. This is obtained from $\lambda$ by successively removing hooks of size $e$ (also called $e$-hooks) until there are no further removable $e$-hooks. We say that $\lambda$ is an $e$-core partition if $\lambda=C_e(\lambda)$. The leg length of a hook is one less than the number of rows it occupies. The $e$-weight of $\lambda$ is given by $w_e(\lambda)=(|\lambda|-|C_e(\lambda)|)/e$. We refer the reader to \cite{JK} or \cite{OlssonBook} for detailed descriptions of these combinatorial objects. 

We record here some useful facts on the degrees of irreducible characters of $\fS_n$. 
Let $p$ be a prime and let $\lambda\in\mathcal{P}(n)$. An immediate consequence of the \textit{hook length formula} \cite[20.1]{James} is that $\chi^\lambda$ has $p$-defect zero if and only if $\lambda$ is a $p$-core partition. At the other end of the spectrum, the set of irreducible characters of $\fS_n$ of degree not divisible by $p$ was completely described in \cite{Mac}. We recall this result in language convenient for our purposes.

\begin{lem}\label{lem:lemma0}
	Let $p$ be a prime and $n\in\N$. Let $n=\sum_{i=1}^ta_ip^{n_i}$ be its $p$-adic expansion, where $n_1>n_2>\cdots >n_t\ge 0$ and $a_i\in[p-1]$ for all $i\in [t]$. Let $\lambda\in\mathcal{P}(n)$ and let $\mu=C_{p^{n_1}}(\lambda)$. Then $\chi^\lambda\in\mathrm{Irr}_{p'}(\fS_n)$ if and only if
	$\mu\in\mathcal{P}(n-a_1p^{n_1})$ and 
	$\chi^{\mu}\in\mathrm{Irr}_{p'}(\fS_{n-a_1p^{n_1}})$.
\end{lem}

\noindent Repeated applications of Lemma~\ref{lem:lemma0} imply the following statement. 

\begin{lem}\label{lem:bigcore}
	Let $n$ be a natural number, let $p$ be a prime and let $\lambda\in\mathcal{P}(n)$.  If $|C_p(\lambda)|\geq p$ then $p$ divides $\chi^\lambda(1)$.
\end{lem}

The Murnaghan--Nakayama rule \cite[2.4.7]{JK} allows us to compute the values of the irreducible characters of $\fS_n$. 
	
\begin{thm}[Murnaghan--Nakayama rule]\label{thm:MNrule}
	Let $r,n\in\N$ with $r<n$. Suppose that $\pi\rho\in\fS_n$ where $\rho$ is an $r$-cycle and $\pi$ is a permutation of the remaining $n-r$ numbers. Then
	\[ \chi^\lambda(\pi\rho) = \sum_\mu (-1)^{h(\lambda\setminus\mu)} \chi^\mu(\pi) \]
	where the sum runs over all partitions $\mu$ obtained from $\lambda$ by removing an $r$-hook, and $h(\lambda\setminus\mu)$ denotes the leg length of the hook removed.
\end{thm}

The cycle types of elements in $\fS_n$ are naturally parametrised by the partitions of $n$. Since the order of cycles is irrelevant, when we refer to cycle types we may sometimes use compositions rather than partitions of $n$. We further remark that if $\sigma\in\fS_n$ has cycle type given by the composition $\alpha=(1^{m_1}2^{m_2}\dotsc)$, then we say that $\sigma$ \textit{contains exactly} $m_i$ $i$-cycles, for all $i\in\mathbb{N}$.
Moreover, its centraliser has size $|\cent{\fS_n}{\sigma}|=\prod_{i\in\N} i^{m_i}\cdot m_i!$.

\subsubsection{Characters and blocks of $\fS_n$ and $\fA_n$}
Let $p$ be a prime. It is well known that the $p$-blocks of $\fS_n$ are parametrised by $p$-core partitions \cite[6.1.21]{JK}. In this article we will denote the $p$-block corresponding to the $p$-core $\gamma$ by $B(\gamma,w)$, where $w$ is the natural number such that $n=|\gamma|+pw$. As explained in \cite[6.2.39]{JK}, defect groups of $B(\gamma,w)$ are Sylow $p$-subgroups of $\fS_{pw}$. Moreover, the set $\Irr_0(B(\gamma,w))$ can be described as follows. 

\begin{lem}\label{lem:ht0}
	Let $n$ be a natural number and let $p$ be a prime. Let $\gamma$ be a $p$-core partition such that $n=|\gamma|+pw$, for some $w\in\mathbb{N}$. Let $pw=\sum_{i=1}^ta_ip^{n_i}$ be its $p$-adic expansion, where $n_1>n_2>\cdots >n_t\ge 0$ and $a_i\in[p-1]$, for all $i\in [t]$. 
	Given $\lambda\in\mathcal{P}(n)$ and $\mu=C_{p^{n_1}}(\lambda)$, we have that $$\chi^\lambda\in\Irr_0(B(\gamma,w))\ \ \text{if and only if}\ \
	\chi^{\mu}\in\Irr_0(B(\gamma,w-a_1p^{n_1-1})).$$
\end{lem}
\begin{proof}
This follows from \cite[Lemma 3.1]{Ols76}.
\end{proof}

In other words, Lemma \ref{lem:ht0} tells us that $\chi^\lambda\in\Irr_0(B(\gamma,w))$ if and only if there exists a sequence of partitions $\lambda=\lambda_0, \lambda_1,\ldots, \lambda_{a_1}=\mu$ such that $\lambda_{i+1}$ is obtained by removing a $p^{n_1}$-hook from $\lambda_i$, and such that $\lambda_{a_1}$ labels an irreducible character of height zero in $B(\gamma, w-a_1p^{n_1-1})$.
When $p=2$, we have the following.

\begin{lem}\label{lem:ht0-not-conj}
Let $n$ be a natural number and let $\lambda\in\mathcal{P}(n)$ be such that $\lambda\neq C_2(\lambda)$. If $\chi^\lambda$ is an irreducible character of height zero in its $2$-block, then $\lambda\neq \lambda'$.
\end{lem}

\begin{proof}
Let $\chi^\lambda\in\mathrm{Irr}(B(\gamma, w))$, for some $2$-core $\gamma$ and some $w\in\mathbb{N}$. 
	Let $2w=2^{n_1}+2^{n_2}+\cdots+2^{n_t}$ be the binary expansion of $2w$ where $t\in\N$ and $n_1>n_2>\cdots>n_t\ge 1$. By Lemma~\ref{lem:ht0}, $\lambda$ has a unique $2^{n_1}$-hook. Assume for a contradiction that $\lambda=\lambda'$. Then $|h_{i,i}(\lambda)|$ is odd for all $i\in\N$. Hence the unique $2^{n_1}$-hook of $\lambda$ is off the main diagonal, i.e.~it is $h_{i,j}(\lambda)$ for some $i\ne j$. This contradicts the assumption that $\lambda=\lambda'$.
\end{proof}

We now briefly recall a description of the irreducible characters and blocks of alternating groups, and refer the reader to \cite[Section 4]{OlssonBlocks} for further detail.
Let $\lambda\in\mathcal{P}(n)$. If $\lambda\ne\lambda'$ then $(\chi^\lambda)_{\fA_n}=(\chi^{\lambda'})_{\fA_n}$ is an irreducible character of $\fA_n$. On the other hand, if $\lambda=\lambda'$ then $(\chi^\lambda)_{\fA_n}=\phi^+_\lambda+\phi^-_\lambda$ with $\phi^{\pm}_\lambda\in\Irr(\fA_n)$. 
All of the irreducible characters of $\fA_n$ are of one of these two forms. 
Turning to blocks, we let $B(\gamma, w)$ be a $p$-block of $\fS_n$ and we first suppose that $p$ is odd.
If $w>0$ then $B(\gamma,w)$ covers a unique block $\hat{B}$ of $\fA_n$. Moreover, $B(\gamma,w)$ and $B(\gamma',w)$ are the only blocks covering $\hat{B}$. If $w=0$ then $B(\gamma,0)$ covers a unique block of $\fA_n$, unless $\gamma=\gamma'$. In the latter case $B(\gamma,0)$ covers two blocks of $\fA_n$ respectively containing the two irreducible constituents of $(\chi^\gamma)_{\fA_n}$.
Finally, if $B(\gamma, w)$ covers $\hat{B}$ then their defect groups are isomorphic. 
On the other hand, if $p=2$ then $\gamma=\gamma'$.
In particular we have that $B(\gamma,w)$ covers a unique block of $\fA_n$ if and only if $w>0$. Moreover, if $D$ is a defect group of $B(\gamma,w)$, then $D\cap\fA_n$ is a defect group of any $2$-block of $\fA_n$ covered by $B(\gamma, w)$.

\medskip

\subsection{Virtual characters of $\fS_n$}

As mentioned in the introduction, the following definition will play a central role in the proof of Theorem B.

\begin{defi}\label{def:C}
	Let $\lambda$ be any partition and let $e\in\N$. Set $n:=|\lambda|+e$. We let $V^\lambda[e]$ be the virtual character of $\fS_n$ defined as follows:
	\[ V^\lambda[e] := \sum_{\alpha} (-1)^{h(\alpha\setminus\lambda)} \chi^\alpha \]
	where $\alpha$ runs over all partitions of $n$ obtained from $\lambda$ by adding an $e$-hook. 
As before, $h(\alpha\setminus\lambda)$ denotes the leg length of the $e$-hook added.
\end{defi}

\begin{ex}\label{eg:C}
	Let $\lambda=(3,1)$ and $e=3$. We observe that $\lambda$ has exactly three addable 3-hooks. In particular, we have that $V^{(3,1)}[3] = \chi^{(6,1)}-\chi^{(3,2,2)}+\chi^{(3,1^4)}$.\hfill$\lozenge$
\end{ex}

We describe the values taken by the virtual characters just introduced. 

\begin{thm}\label{thm:gdc}
	Let $\lambda$ be any partition and let $e\in\N$. Let $n=|\lambda|+e$ and let $\sigma\in\fS_n$. 
	Suppose that the disjoint cycle decomposition of $\sigma$ contains exactly $k$ $e$-cycles. 
	Then
	\[ V^\lambda[e](\sigma) = \begin{cases}
	ke\cdot\chi^\lambda(\tau) & \mathrm{if}\ k>0,\\
	0 & \mathrm{if}\ k=0,
	\end{cases}\]
 	where $\tau\in\fS_{n-e}$ has cycle type equal to that of $\sigma$ except with one fewer $e$-cycle.
\end{thm}

We observe that Theorem \ref{thm:gdc} extends \cite[Theorem 21.7]{James}. The result may be known to experts in the field, but we could not find it in the literature.
To prove Theorem~\ref{thm:gdc}, we use results from \cite{Stanley}, translating between the language of symmetric polynomials and class functions. We briefly summarise here the relevant notation.

For a partition $\mu$, we let $s_\mu$ denote the corresponding Schur function. For $e\in\mathbb{N}$, $p_e$ denotes the power sum symmetric function $\sum_i x_i^e$ in indeterminates $x_i$. If $\mu=(\mu_1,\mu_2,\dotsc,\mu_k)$, then $p_\mu$ is defined to be the product $p_{\mu_1}p_{\mu_2}\cdots p_{\mu_k}$.
The Frobenius characteristic map $\ch$ is a ring isomorphism between the algebra of class functions of finite symmetric groups and the ring of symmetric functions (for more detail, see \cite[\textsection 7.18]{Stanley}).
If $f$ is a class function of $\fS_n$, then
\begin{equation}\label{eq:ch}
\ch(f) = \sum_{\mu\vdash n}z_\mu^{-1}f(\mu)p_\mu,
\end{equation}
where $z_\mu=|\cent{\fS_n}{\omega}|$ with $\omega\in\fS_n$ an element of cycle type $\mu$. In particular, $\ch(\chi^\lambda)=s_\lambda$, for any partition $\lambda$. Moreover, for all class functions $f$ of $\fS_m$ and $g$ of $\fS_n$ we have that $\ch$ satisfies $\ch(f\circ g) = \ch(f)\cdot \ch(g)$. Here $f\circ g$ denotes the induced class function $(f\times g)^{\fS_{m+n}}_{\fS_m\times\fS_n}$. 

\begin{proof}[Proof of Theorem~\ref{thm:gdc}]
	From \cite[Theorem 7.17.1]{Stanley}, we have that
	\[ s_\lambda\cdot p_e = \sum_\alpha(-1)^{h(\alpha\setminus\lambda)}s_\alpha \]
	where the sum runs over all partitions $\alpha$ obtained from $\lambda$ by adding an $e$-hook.
	It is easy to see from \eqref{eq:ch} that if $\ch(f)=p_e$, then $f$ is the class function of $\fS_e$ given by
	\[ f(\omega) = \begin{cases}
	e & \text{if the cycle type of $\omega$ is }(e),\\
	0 & \text{otherwise}.
	\end{cases} \]
	It follows that
	\[ \ch(\chi^\lambda\circ f)=\ch(\chi^\lambda)\cdot \ch(f) = s_\lambda\cdot p_e = \sum_\alpha(-1)^{h(\alpha\setminus\lambda)}\ch(\chi^\alpha). \]
	Since $\ch$ is bijective and linear, we have that $\chi^\lambda\circ f = V^\lambda[e].$ Thus it remains to prove that
	\[  (\chi^\lambda\circ f)(\sigma)=\begin{cases}
	ke\cdot\chi^\lambda(\tau) & \text{if}\ k>0,\\
	0 & \text{if}\ k=0,\end{cases}\]
where $\tau\in\fS_{n-e}$ has cycle type equal to that of $\sigma$ except with one fewer $e$-cycle.
Since $\chi^\lambda\circ f=(\chi^\lambda\times f)^{\fS_{n+e}}$, this follows directly from \cite[(5.1)]{Isa06} and the definition of $f$ given above.
\end{proof}

We now extend Definition~\ref{def:C} by allowing the addition of multiple hooks.

\begin{defi}\label{def:iteratedC}
	Let $\lambda$ be any partition. For $e,f\in\N$, define
	\[ V^\lambda[e,f] := \sum_{\alpha}(-1)^{h(\alpha\setminus\lambda)} \sum_\beta (-1)^{h(\beta\setminus\alpha)} \chi^\beta. \]
	Here $\alpha$ runs over all partitions obtained from $\lambda$ by adding an $e$-hook. For each fixed such $\alpha$, we have $\beta$ running over partitions obtained from $\alpha$ by adding an $f$-hook.
	
	We define $V^\lambda[e_1,e_2,\dotsc,e_u]$ analogously for any $u\in\N\cup\{0\}$ and any sequence of natural numbers $e_1,\dotsc,e_u$. 
	If $u=0$, we set $V^\lambda[e_1,\dotsc,e_u]:=\chi^\lambda$.
\end{defi}

Observe that $V^\lambda[e,f]=\sum_{\alpha}(-1)^{h(\alpha\setminus\lambda)}V^{\alpha}[f]$. It is then easy to see that Theorem \ref{thm:gdc} implies $V^\lambda[e,f]=V^\lambda[f,e]$.
Similarly, we have that 
$V^\lambda[e_1,\dotsc,e_u] = V^\lambda[e_{\rho(1)},\dotsc,e_{\rho(u)}]$
for any $\rho\in\fS_u$.

\begin{ex}\label{eg:iteratedC}
	Following on from Example~\ref{eg:C}, we can compute $V^{(3,1)}[3,3]$, a virtual character of $\fS_{10}$:
	\begin{align*}
	V^{(3,1)}[3,3] &= (-1)^0 \cdot V^{(6,1)}[3] + (-1)^1\cdot V^{(3,2,2)}[3] + (-1)^2\cdot V^{(3,1^4)}[3]\\
	&= \chi^{(9,1)} + \chi^{(6,4)} -2\chi^{(6,2,2)} + 2\chi^{(6,1^4)} +\chi^{(4,4,2)} + 2\chi^{(3,2^3,1)}\\
	&\qquad\qquad  -\chi^{(3,2,2,1^3)} -\chi^{(3,3,2,1,1)} + \chi^{(3,1^7)}.\qquad\qquad\qquad\qquad\quad\lozenge
	\end{align*}
\end{ex}

We now turn to the study of the restriction to Sylow $p$-subgroups of the virtual characters introduced in Definition \ref{def:iteratedC}.
As mentioned at the beginning of Section \ref{sec:coprime}, we will use the symbol $P_n$ to denote a fixed Sylow $p$-subgroup of $\fS_n$.

\begin{thm}\label{thm:p'mult}
	Let $p$ be a prime, let $n\in\N$ and let $\gamma$ be a $p$-core partition such that $|\gamma|\le n$ and $p\mid n-|\gamma|$. Let $u\in\N\cup\{0\}$ and $t_1,\dotsc,t_u\in\mathbb{N}$ be such that 
	$n-|\gamma|=p^{t_1}+\cdots+p^{t_u}$.
	Then 
	\[ p\nmid \big[V^\gamma[p^{t_1},\dotsc,p^{t_u}]_{P_n}, {\bf 1}_{P_n}\big]. \]
\end{thm}
\begin{proof}
	If $|\gamma|=n$ then $u=0$ and $V^\gamma[p^{t_1},\dotsc,p^{t_u}]=\chi^\gamma$ is a $p$-defect zero character. Then $p$ does not divide  $[(\chi^\gamma)_{P_n},{\bf 1}_{P_n}]$ by Lemma~\ref{lem:defectzero}. 
	
	We now assume that $|\gamma|<n$ and $u>0$. 
	To ease the notation we let $e_i=p^{t_i}$, for all $i\in [u]$.
	We recall that $\chi^\gamma$ is zero on every non-trivial element of $P_{|\gamma|}$ since $\gamma$ is a $p$-core partition. By repeated applications of Theorem~\ref{thm:gdc}, we find that $V^\gamma[e_1,\dotsc,e_u]$ is zero on every element of $P_n$ except those with cycle type $(e_1,e_2,\dotsc,e_u,1^{|\gamma|})$. Let $\sigma\in P_n$ have cycle type $(e_1,e_2,\dotsc,e_u,1^{|\gamma|})$. Let $T_\sigma$ be the class function of $P_n$ taking value 1 on those elements with the same cycle type as $\sigma$ and 0 otherwise. Then $V^\gamma[e_1,\dotsc,e_u]_{P_n}=zT_\sigma$ for some $z\in\N$. Thus
	\[ \big[ V^\gamma[e_1,\dotsc,e_u]_{P_n},{\bf 1}_{P_n}\big] = [zT_\sigma, {\bf 1}_{P_n}] = \frac{1}{|P_{n}|} \sum_{\tau\in P_{n}} zT_\sigma(\tau) = \frac{z}{|P_n|}\cdot|P_n\cap\sigma^{\fS_n}|. \]
	On the other hand, from Theorem~\ref{thm:gdc} we also have that
	\[ V^\gamma[e_1,\dotsc,e_u](\sigma) = \chi^\gamma(1)\cdot\prod_{i\ge 2} (i^{a_i}\cdot a_i!) \]
	where $a_i=|\{j\in[u] : e_j=i\}|$. Hence
	\[ z =zT_\sigma(\sigma)= V^\gamma[e_1,\dotsc,e_u](\sigma) = \chi^\gamma(1)\cdot\frac{|\cent{\fS_n}{\sigma}|}{|\gamma|!}. \]
	Therefore, to conclude the proof it suffices to show that
	\begin{equation}\label{eq:stp}
	\nu_p\big(\chi^\gamma(1)\big) - \nu_p(|\gamma|!)+\nu_p\big(|\cent{\fS_n}{\sigma}|\big) - \nu_p(|\fS_n|)+\nu_p\big(|P_n\cap\sigma^{\fS_n}|\big)=0,
	\end{equation}
	where we have used that $\nu_p(|P_n|)=\nu_p(|\fS_n|)$. Since $\gamma$ is a $p$-core, we have $\nu_p\big(\chi^\gamma(1)\big) = \nu_p(|\gamma|!)$. Thus \eqref{eq:stp} follows from the Orbit--Stabiliser theorem (giving $|\fS_n|=|\sigma^{\fS_n}|\cdot|\cent{\fS_n}{\sigma}|$) and Lemma~\ref{lem:iv}.
\end{proof}

Starting with Definition \ref{def:C}, in this section we introduced a family of virtual characters of $\fS_n$ and we have studied their properties. 
The main ingredient of our proof of Theorem B for symmetric groups (Corollary \ref{cor:B-Sn} below) is a specific member of this family. 
We highlight this specific virtual character in the following definition. 

\begin{defi}\label{def: VirtualCharBlock}
Let $\gamma$ be a $p$-core partition and let $n=|\gamma|+wp$ for some integer $w>0$. Let $B=B(\gamma, w)$ be the $p$-block of $\fS_n$ labelled by $\gamma$. 
Let $wp=\sum_{i\ge 1}a_ip^i$ with $a_i\in\{0,1,\dotsc,p-1\}$ for each $i$.
We denote by $V^B$ the virtual character of $\fS_n$ defined as 
\[ V^B=V^\gamma[e_1,\dotsc,e_u], \]
where $e_1,\dotsc,e_u$ are natural numbers such that
$|\{j\in[u] : e_j=p^i \}|=a_i$ for all $i\ge 1$, and $u=\sum_i a_i$.
\end{defi}

In other words, the numbers $e_1,\dotsc,e_u$ in Definition~\ref{def: VirtualCharBlock} are the various $p^i$ appearing in the $p$-adic expansion $wp=\sum_{i\ge 1} a_ip^i$, counted with multiplicity.

\smallskip

We remark that every irreducible character $\chi\in\Irr(\fS_n)$ appearing with non-zero coefficient in $V^B$ belongs to $\Irr_0(B)$. This follows directly from Lemma~\ref{lem:ht0}. 

\begin{cor}\label{cor:B-Sn}
	Theorem B holds when $G$ is a finite symmetric group.
\end{cor}

\begin{proof}
	Let $G=\fS_n$ and suppose $B=B(\gamma,w)$ for some $p$-core $\gamma$ and $w\ge 0$. If $w=0$ then $\chi^\gamma\in\Irr_0(B)$ since $\gamma$ is a $p$-core partition. By Lemma~\ref{lem:defectzero}, $p\nmid [(\chi^\gamma)_{P_n},{\bf 1}_{P_n}]$.
	
	Now suppose $w>0$. 
	Consider the virtual character $V^B=V^\gamma[e_1,\dotsc,e_u]$, introduced in Definition \ref{def: VirtualCharBlock}, and note that $e_i>1$ and $e_i$ is a power of $p$ for every $i\in[u]$.
	It follows from Theorem~\ref{thm:p'mult} that $p\nmid[(V^B)_{P_n},{\bf 1}_{P_n}]$. 
	Hence, there exists an irreducible character $\chi$ occurring in $V^B$ that satisfies $p\nmid[\chi_{P_n},{\bf 1}_{P_n}]$.
	As remarked after Definition \ref{def: VirtualCharBlock} we know that $\chi\in\Irr_0(B)$, as desired. 
\end{proof}

We are now ready to treat the case of alternating groups, which will conclude the proof of Theorem B. 

\begin{cor}\label{cor:B-An}
	Theorem B holds when $G$ is a finite alternating group.
\end{cor}

\begin{proof}
	Let $G=\fA_n$ and suppose $B$ is a $p$-block of $\fA_n$. Let $\bar{B}=B(\gamma,w)$ be a block of $\fS_n$ covering $B$, and let $\bar{D}$ and $D=\bar{D}\cap\fA_n$ be defect groups of $\bar{B}$ and $B$ respectively. Let $\bar{P}\in\syl{p}{\fS_n}$ and $P:=\bar{P}\cap\fA_n\in\syl{p}{\fA_n}$. 
	
	First assume $p$ is odd. Then $P=\bar{P}$ and $D=\bar{D}$. 
	By Corollary~\ref{cor:B-Sn}, there exists $\chi^\lambda\in\Irr_0(\bar{B})$ such that $p\nmid[(\chi^\lambda)_{\bar{P}},{\bf 1}_{\bar{P}}]$. 
	If $\lambda\ne\lambda'$ then $\phi:=(\chi^\lambda)_{\fA_n}\in B$. Since $\chi^\lambda\in\Irr_0(\bar{B})$ we have that $\phi\in\Irr_0(B)$, as $|\bar{D}|=|D|$ and $p$ is odd. Moreover, $p\nmid[\phi_P,{\bf 1}_P]=[(\chi^\lambda)_{\bar{P}},{\bf 1}_{\bar{P}}]$. 
	On the other hand, if $\lambda=\lambda'$ then $(\chi^\lambda)_{\fA_n}=\phi^+_\lambda + \phi^-_\lambda$ and at least one of $\phi^{+}_\lambda$ and $\phi^{-}_\lambda$ belongs to $B$. Note $\phi^{+}_\lambda(1)=\phi^{-}_\lambda(1)=\tfrac{1}{2}\chi^\lambda(1)$. Since $p$ is odd we have that both $\phi^{+}_\lambda$ and $\phi^{-}_\lambda$ are height zero characters in their block. Moreover, $p\nmid[(\phi^{+}_\lambda)_P,{\bf 1}_P]=[(\phi^{-}_\lambda)_P,{\bf 1}_P]=\tfrac{1}{2}[(\chi^\lambda)_{\bar{P}},{\bf 1}_{\bar{P}}]$.
	
	Now assume $p=2$. First suppose $w>0$. Let $2w=2^{n_1}+2^{n_2}+\cdots+2^{n_t}$ be the binary expansion of $2w$ where $t\in\N$ and $n_1>\cdots>n_t\ge 1$. Let $V^{\bar{B}}=V^\gamma[2^{n_1},\ldots,2^{n_t}]$. By Lemma~\ref{lem:ht0}, every irreducible character $\chi^\lambda\in\Irr(\fS_n)$ occurring in the linear combination $V^{\bar{B}}$ belongs to $\Irr_0(\bar{B})$. Moreover, $\lambda\ne\lambda'$ for all such characters $\chi^\lambda$ by Lemma~\ref{lem:ht0-not-conj}. Since $\gamma$ is self-conjugate, if $\chi^\lambda$ occurs in $V^{\bar{B}}$ then so does $\chi^{\lambda'}$. In particular, we can write $V^{\bar{B}}$ as a sum of terms of the form $\pm(\chi^\lambda+\chi^{\lambda'})$ or $(\chi^\lambda-\chi^{\lambda'})$ for various $\lambda$. Since $2\nmid[(V^{\bar{B}})_{\bar{P}},{\bf 1}_{\bar{P}}]$ by Theorem~\ref{thm:gdc}, we deduce that $2\nmid [(\chi^\lambda)_{\bar{P}},{\bf 1}_{\bar{P}}]+[(\chi^{\lambda'})_{\bar{P}},{\bf 1}_{\bar{P}}]$ for some $\chi^\lambda\in\Irr_0(\bar{B})$. Let $\phi:=(\chi^\lambda)_{\fA_n}$. Since $|D|=\tfrac{1}{2}|\bar{D}|$, we deduce that $\phi\in\Irr_0(B)$. Moreover, $$[\phi_P,{\bf 1}_P]=[(\chi^\lambda)_{\bar{P}},{\bf 1}_{\bar{P}}]+[(\chi^{\lambda'})_{\bar{P}},{\bf 1}_{\bar{P}}].$$
We conclude that $2\nmid [\phi_P,{\bf 1}_P]$, as desired. 
	
	Finally, if $p=2$ and $w=0$, then $D=1$ and $B$ contains a unique irreducible character $\phi$ that is at the same time of $p$-defect zero and of height zero in $B$. By Lemma~\ref{lem:defectzero} we have that 
$2\nmid [\phi_P,{\bf 1}_P]$. The proof is concluded.
\end{proof}

As promised at the start of Section \ref{sec:coprime}, we use Theorem B to prove Theorem~\ref{thm:An23} and thereby complete the proof of Theorem A. In order to do this, we first show that irreducible characters in symmetric and alternating groups of degree coprime to $p$ are characterised by the non-vanishing property on Sylow $p$-subgroups. 

\begin{thm}\label{thm:nonvanishing} 
	Let $G$ be a finite symmetric or alternating group, $p$ be a prime and $P\in \syl p G$. Then $\chi \in \irr G$ has degree coprime to $p$ if and only if $\chi(x)\neq 0$ for every $x \in P$.
\end{thm}

\begin{proof}
	Let $\chi\in\Irr(G)$. By Lemma~\ref{lem:nonvanishing2}, we need to prove that if $\chi(x)\ne 0$ for every $x\in P$ then $p\nmid\chi(1)$. 
	
	Let $n\in\N$ with $p$-adic expansion $n=\sum_{i=1}^ta_ip^{n_i}$, where $n_1>n_2>\cdots>n_t\ge 0$ and $a_i\in[p-1]$ for all $i\in [t]$. 
	Since the theorem holds trivially for $n<p$, we assume from now on that $n\geq p$.
	We call an element $g\in\fS_n$ a \emph{$p$-adic element} if in the disjoint cycle decomposition of $g$ there are $a_i$ cycles of length $p^{n_i}$ for each $i\in[t]$. 
	We proceed by splitting the proof into two cases according to $G=\fS_n$ or $G=\fA_n$.
	
	\smallskip

	\noindent\textbf{(i) $G=\fS_n$:}  Given $\chi\in\mathrm{Irr}(\fS_n)$ with $p\mid \chi(1)$, we claim that $\chi(g)=0$ for any $p$-adic element $g\in P$.
We show that the above claim holds by induction on $t$, the $p$-adic length of $n$. 
Let $\chi=\chi^\lambda$ for some $\lambda\vdash n$. 
	
	If $t=1$, then $n=ap^k$ for some $a\in[p-1]$ and $k\in\N$. In this setting $g$ has cycle type $(p^k,p^k,\ldots, p^k)$ (i.e.~$g$ is the product of $a$ cycles of length $p^k$).
	Since $p$ divides $\chi^\lambda(1)$, by Lemma \ref{lem:lemma0} we have that $|C_{p^k}(\lambda)|>0$. Equivalently, the $p^k$-weight of $\lambda$ is strictly smaller than $a$. Hence it is not possible to successively remove $a$ $p^k$-hooks from $\lambda$. Using the Murnaghan--Nakayama rule, we conclude that $\chi^\lambda(g)=0$. 
	
	Let us now assume that $t\geq 2$ and that the claim holds for $t-1$. In this setting we have that $g=\rho\pi$, where $\rho$ is the product of $a_1$ cycles of length $p^{n_1}$ and $\pi$ is a $p$-adic element of $\fS_{n-a_1p^{n_1}}$. Clearly, the $p^{n_1}$-weight $w_{p^{n_1}}(\lambda)$ of $\lambda$ is smaller than or equal to $a_1$. If $w_{p^{n_1}}(\lambda)<a_1$ then $\chi^\lambda(g)=0$ by the Murnaghan--Nakayama rule. Otherwise, $w_{p^{n_1}}(\lambda)=a_1$ and $\nu:=C_{p^{n_1}}(\lambda)$ is a partition of $n-a_1p^{n_1}$. Since $p$ divides $\chi^\lambda(1)$, Lemma \ref{lem:lemma0} implies that $p\mid\chi^\nu(1)$. The inductive hypothesis now guarantees that $\chi^\nu(\pi)=0$. This concludes the proof of our claim, as another application of the Murnaghan--Nakayama rule shows that there exists $k\in\Z$ such that $\chi^\lambda(g)=k\cdot\chi^\nu(\pi)=0$. 
	
	We can now conclude that if $\chi^\lambda(x)\neq 0$ for all $x\in P$ then $\chi^\lambda\in\mathrm{Irr}_{p'}(\fS_n)$, because we can always find a $p$-adic element of $\fS_n$ lying in $P$.

	\smallskip
	
	\noindent\textbf{(ii) $G=\fA_n$:} 
	If $p$ is odd then $P\in\syl p{\fA_n}$ is also a Sylow $p$-subgroup of $\fS_n$, and hence $P$ contains a $p$-adic element $g$. 
	Consider $\phi\in\Irr(\fA_n)$ with $p\mid\phi(1)$, and let $\lambda\vdash n$ be such that $\phi$ is an irreducible constituent of $(\chi^\lambda)_{\fA_n}$. If $\lambda\neq \lambda'$, then $\phi=(\chi^\lambda)_{\fA_n}$ and hence $\phi(g)=\chi^\lambda(g)=0$, as $p$ divides $\chi^\lambda(1)$. 
	If $\lambda=\lambda'$ then $p$ divides $\chi^\lambda(1)=2\phi(1)$. Let $\delta:=(|h_{1,1}(\lambda)|, |h_{2,2}(\lambda)|,\ldots, |h_{\ell,\ell}(\lambda)|)$ be the partition of $n$ given by the diagonal hook lengths of $\lambda$. We observe that $\delta$ cannot be equal to the cycle type of $g$. This is clear whenever $a_i\geq 2$ for some $i\in[t]$, because $|h_{j,j}(\lambda)|>|h_{j+1,j+1}(\lambda)|$ for all $j\in [\ell-1]$. On the other hand, if $a_i=1$ for all $i\in[t]$, then Lemma~\ref{lem:lemma0} shows that we cannot have $|h_{i,i}(\lambda)|=p^{n_i}$ for all $i\in[t]$, because $p$ divides $\chi^\lambda(1)$.
	It follows that $\phi(g)=\frac{1}{2}\chi^\lambda(g)=0$, by \cite[2.5.13]{JK}. 

	If $p=2$, then $P$ is a subgroup of index $2$ of a Sylow $2$-subgroup of $\fS_n$. 
	Let $\phi\in\Irr(\fA_n)$ with $2\mid\phi(1)$ and let $\lambda\vdash n$ be such that $\phi$ is an irreducible constituent of $(\chi^\lambda)_{\fA_n}$.
	We observe that $\chi^\lambda(1)$ is even. From Lemma \ref{lem:lemma0} we deduce that there exists $s\in[t]$ such that $|C_{2^s}(\lambda)|>n-(2^{n_1}+\cdots+2^{n_s})$. 
	Let $$r=\mathrm{min}\{s\in[t]\ :\ |C_{2^s}(\lambda)|>n-(2^{n_1}+\cdots+2^{n_s})\}.$$
	We now pick $g\in \fS_n$ of cycle type 
	\[ \begin{cases}
	(2^{n_1}, 2^{n_2}, \ldots, 2^{n_t}) & \text{if } t \text{ is even},\\[5pt]
	(2^{n_1}, 2^{n_2-1}, 2^{n_2-1}, 2^{n_3}, \ldots, 2^{n_t}) & \text{if } t \text{ is odd and }r=1,\\[5pt]
	(2^{n_1-1}, 2^{n_1-1}, 2^{n_2}, 2^{n_3}, \ldots, 2^{n_t}) &\text{if } t \text{ is odd and }r>1.
	\end{cases}\]
	
	We observe that such a $g$ is an even permutation and therefore can be found in $P$. Moreover, we have that $\chi^\lambda(g)=0$ by the Murnaghan--Nakayama rule. It follows that if $\lambda\neq \lambda'$ then $\phi(g)=\chi^\lambda(g)=0$.
	Finally, if $\lambda=\lambda'$ then clearly $|h_{i,i}(\lambda)|$ is odd for all $i\in[t]$, hence we cannot have $h_{i,i}(\lambda)=2^{n_i}$ for all $i\in[t]$. Therefore $\phi(g)=\frac{1}{2}\chi^\lambda(g)=0$, by \cite[2.5.13]{JK}.
\end{proof}

\begin{thm}\label{thm:An23}
	Let $p\in\{2,3\}$ and let $n\ge 5$ be a natural number. Let $R$ be a Sylow $p$-subgroup of $\fA_n$. Then	$\fA_n$ either possesses a $p$-defect zero character or there is some  $\aut {\fA_n}$-invariant $\theta \in \irr {\fA_n}$ such that $[\theta_R, {\bf 1}_R]$ is not divisible by $p$ and $\theta(x)=0$ for some $x \in R$.
\end{thm}

\begin{proof}
	From Theorem~\ref{thm:nonvanishing} we know that for $\theta \in \irr {\fA_n}$ the condition $\theta(x)=0$ for some $x \in R$ is equivalent to having that $p$ divides $\theta(1)$. 
	We also recall that for every $n\neq 6$, we have that $\theta$ is $\aut {\fA_n}$-invariant if and only if $\theta=(\chi^\lambda)_{\fA_n}$ for some partition $\lambda$ of $n$ such that $\lambda\neq \lambda'$. 

	We distinguish two cases, depending on the value of $p\in\{2,3\}$. 

	\noindent\textbf{(i) $p=2$:} If $n=6$ then both constituents of $(\chi^{(3,2,1)})_{\fA_6}$ are $2$-defect zero irreducible characters of $\fA_6$. Suppose now that $n\ne 6$. Let $\gamma$ be the following $2$-core partition:
	\[ \gamma=\begin{cases}
	(2,1) & \mathrm{if}\ n\ \text{is odd},\\
	(3,2,1) &\mathrm{if}\ n\ \text{is even}.\end{cases}\]
	Let $2w=n-|\gamma|$ and let $2w=2^{n_1}+2^{n_2}+\cdots+2^{n_t}$ be its binary expansion, where $t\in\mathbb{N}$ and $n_1>n_2>\cdots >n_t\geq 1$.
	Let $B:=B(\gamma, w)$ be the corresponding $2$-block of $\fS_n$. By Lemma~\ref{lem:ht0-not-conj}, if $\chi^\lambda\in\Irr_0(B)$ then $\lambda\neq \lambda'$. 
	Considering as in the proof of Corollary~\ref{cor:B-An} the virtual character $V^{\gamma}[2^{n_1},2^{n_2},\ldots,2^{n_t}]$, we deduce that there exists $\lambda\in\mathcal{P}(n)$ labelling a character in $\Irr_0(B)$ such that the sum of Sylow branching coefficients $Z^\lambda+Z^{\lambda'}$ is odd. Since $|\gamma|\geq 3$ we also have that $\chi^\lambda(1)$ is even, by Lemma \ref{lem:bigcore}. 
	We conclude that $\theta=(\chi^\lambda)_{\fA_n}$ is an irreducible character of $\fA_n$ with the desired properties, since $[\theta_R,{\bf 1}_R] = Z^\lambda+Z^{\lambda'}$.

	\smallskip

	\noindent\textbf{(ii) $p=3$:} Direct verification shows that the statement holds for $n\in\{5,8,11\}$. More precisely, the alternating groups $\fA_5$ and $\fA_8$ admit $3$-defect zero irreducible characters. On the other hand, $\theta=(\chi^{(8,2,1)})_{\fA_{11}}$ is such that $3\mid\theta(1)$ and $3\nmid[\theta_R,{\bf 1}_R]$ where $R\in\syl{3}{\fA_{11}}$.
	
	Now suppose that $n\notin\{5,8,11\}$. Let $\gamma$ be the following $3$-core partition:
	\[ \gamma=\begin{cases}
	(4,2) & \mathrm{if}\ n\equiv 0\ (\mathrm{mod}\ 3),\\
	(3,1) & \mathrm{if}\ n\equiv 1\ (\mathrm{mod}\ 3),\\
	(6,4,2,1,1) &\mathrm{if}\ n\equiv 2\ (\mathrm{mod}\ 3).\end{cases}\]
	Let $B$ be the $3$-block of $\fS_n$ labelled by $\gamma$. Since $\gamma\neq \gamma'$ we have that every irreducible character in $\Irr_0(B)$ is labelled by a non-self-conjugate partition. By Theorem B there exists $\chi^\lambda\in\Irr_0(B)$ such that $Z^\lambda$ is coprime to $3$. Since $|\gamma|\geq 4$ we also have that $3$ divides $\chi^\lambda(1)$, by Lemma \ref{lem:bigcore}. We conclude that $\theta=(\chi^\lambda)_{\fA_n}$ is an irreducible character of $\fA_n$ with the desired properties. 
\end{proof}

\medskip

\noindent\textbf{Acknowledgments.}
Part of this work was done while the fourth author was visiting the first at the University of Florence
supported 
by the Spanish National Research Council through the ``Ayuda
extraordinaria a Centros de Excelencia Severo Ochoa'' (20205CEX001).
The third author was supported by ERC Consolidator Grant 647678. The fourth author is partially supported by the Spanish Ministerio de Ciencia e Innovaci\'on
PID2019-103854GB-I00 and FEDER funds.
We thank Gabriel Navarro and Thomas Breuer for helping us checking Theorem 2.3 for some sporadic groups.  Moreover, we would like to
thank Gunter Malle for comments on a previous version.

\end{document}